\documentclass[a4paper]{amsart}
\usepackage{orcidlink}

\title{Tensor factorization of local algebras}

\author[Taro Sakurai]{Taro Sakurai\,\orcidlink{0000-0003-0608-1852}}
\address[Taro Sakurai]
{Department of Mathematics and Informatics, Graduate School of Science, Chiba University, 1-33 Yayoi-cho, Inage-ku, Chiba-shi, Chiba, 263-8522, Japan.}
\email{tsakurai@math.s.chiba-u.ac.jp}

\date{\today}



\usepackage[T1]{fontenc}
\usepackage{lmodern}

\usepackage[svgnames]{xcolor}
\usepackage{tikz-cd}
\usepackage{mathtools}

\usepackage{aliascnt} 
\usepackage{hyperref}
\makeatletter
\hypersetup{
  colorlinks=true,
  allcolors=MediumBlue,
  pdflang=en,
  pdfauthor=\shortauthors, 
  pdftitle=\@title,
  pdfsubject=\@subjclass,
  pdfkeywords=\@keywords,
}
\makeatother
\usepackage{imakeidx}

\usepackage[backend=biber,giveninits=true,uniquename=false,doi=true]{biblatex}
\renewbibmacro{in:}{}
\DeclareFieldFormat{pages}{#1}

\usepackage{cleveref}
\crefdefaultlabelformat{#2\textup{#1}#3}
\crefname{section}{Section}{Sections}
\crefname{mainthm}{Theorem}{Theorems}
\crefname{maincor}{Corollary}{Corollaries}
\crefname{thm}{Theorem}{Theorems}
\crefname{prop}{Proposition}{Propositions}
\crefname{lem}{Lemma}{Lemmas}
\crefname{rmk}{Remark}{Remarks}
\crefname{dfn}{Definition}{Definitions}
\crefname{hyp}{Hypothesis}{Hypotheses}

\theoremstyle{plain}

\newtheorem{mainthm}{Theorem}

\newaliascnt{maincor}{mainthm}
\newtheorem{maincor}[maincor]{Corollary}
\aliascntresetthe{maincor}

\newaliascnt{prop}{thm}
\newtheorem{prop}[prop]{Proposition}
\aliascntresetthe{prop}

\newaliascnt{lem}{thm}
\newtheorem{lem}[lem]{Lemma}
\aliascntresetthe{lem}

\theoremstyle{definition}

\newaliascnt{rmk}{thm}
\newtheorem{rmk}[rmk]{Remark}
\aliascntresetthe{rmk}

\newaliascnt{dfn}{thm}
\newtheorem{dfn}[dfn]{Definition}
\aliascntresetthe{dfn}

\newaliascnt{hyp}{thm}
\newtheorem{hyp}[hyp]{Hypothesis}
\aliascntresetthe{hyp}

\newcommand{\FF}{\mathbb{F}}
\newcommand{\QQ}{\mathbb{Q}}
\renewcommand{\epsilon}{\varepsilon}
\renewcommand{\phi}{\varphi}
\DeclareMathOperator{\id}{id}
\DeclareMathOperator{\Ker}{Ker}
\renewcommand{\Im}{\operatorname{Im}}
\newcommand{\Elem}{\textup{E}}
\newcommand{\Ab}{\textup{A}}

\begin{document}

\begin{abstract}
  Let $A$ be a finite-dimensional local algebra over a splitting field.
  We prove that an irreducible tensor factorization $A \cong P_1 \otimes \dotsb \otimes P_n$ of the algebra $A$ is unique up to isomorphism and order of $P_1, \dotsc, P_n$.
\end{abstract}

\maketitle

\section{Introduction}
Let \( F \) be a field.
\index{\( F \) : Field}%
All algebras are assumed to be associative and finite-dimensional over \( F \) with identity \( 1 \neq 0 \).
All tensor products \( \otimes \) are taken over \( F \).
An algebra is trivial if it is isomorphic to \( F \).
For algebras \( A \) and \( B \), we call \( B \) a tensor factor of \( A \) if \( A \cong B \otimes C \) for some algebra \( C \).
\index{\( A \) : Algebra}%
\index{\( B \) : Algebra}%
\index{\( C \) : Algebra}%
A non-trivial algebra \( P \) is \emph{irreducible} if every tensor factor of \( P \) is isomorphic to \( F \) or \( P \).
\index{\( P \) : Algebra}%
A tensor factorization
\[
  A \cong P_1 \otimes \dotsb \otimes P_n
\]
of an algebra \( A \) is \emph{irreducible} if \( P_1, \dotsc, P_n \) are irreducible.
\index{\( n \) : Integer}%
\index{\( P_1, \dotsc, P_n \) : Algebra}%
The empty tensor product is understood to be \( F \).
Every algebra admits an irreducible tensor factorization, but it need not be unique.

For example, the number field \( \QQ(\sqrt{2}, \sqrt{3}) \) admits two irreducible tensor factorizations
\[
  \QQ(\sqrt{2}, \sqrt{3})
  \cong \QQ(\sqrt{2}) \otimes \QQ(\sqrt{6})
  \cong \QQ(\sqrt{3}) \otimes \QQ(\sqrt{6}),
\]
but \( \QQ(\sqrt{2}) \) and \( \QQ(\sqrt{3}) \) are non-isomorphic.
\index{\( \QQ \) : Field}%
In light of the results below, uniqueness fails here because \( \QQ(\sqrt{2}, \sqrt{3}) \) does not split over \( \QQ \).

Nevertheless, the tensor product behaves better for local algebras over splitting fields.
Horst's theorem~\cite{Ho87} shows uniqueness of irreducible tensor factorization for local (commutative\footnote{See \cref{rmk:comm}.}) algebras over splitting fields of characteristic zero.
It is essential in her proof to work in characteristic zero, because removing that assumption from a lemma gives rise to counterexamples\footnote{See \cref{rmk:char}.}.
As N\"usken \cite[p.~527]{Nu02} pointed out, the problem remained open in positive characteristic.

In this paper, we establish uniqueness of irreducible tensor factorization for local algebras over splitting fields without assuming commutativity or characteristic zero.

\begin{mainthm}
  \label{main:thm}
  Let \( A \) be a finite-dimensional local algebra over a splitting field.
  If \( A \) admits two irreducible tensor factorizations
  \[
    A \cong P_1 \otimes \dotsb \otimes P_n \cong Q_1 \otimes \dotsb \otimes Q_m,
  \]
  then \( n = m \) and, after reordering indices, \( P_i \cong Q_i \) for every \( i \).
\end{mainthm}
\index{\( m \) : Integer}%
\index{\( Q_1, \dotsc, Q_m \) : Algebra}%
\index{\( i \) : Integer}%

The uniqueness implies that the tensor product is cancellative for local algebras over splitting fields.

\begin{maincor}
  \label{main:cor}
  Let \( A, B, C \) be finite-dimensional local algebras over a splitting field.
  Then
  \[
    A \otimes C \cong B \otimes C \implies A \cong B.
  \]
\end{maincor}

After recalling basic definitions in \cref{sec:prelim}, we prove a key tensor factorization in \cref{prop:factor}.
Then we show how to match the first tensor factor in \cref{prop:first} and the other tensor factors in \cref{prop:other}.
In \cref{sec:uniq}, we present the inductive proof of uniqueness.
Finally, we provide an application of the cancellation law to the modular isomorphism problem in \Cref{appx:MIP}.

\section{Preliminaries}
\label{sec:prelim}
Let \( A \) be an algebra over a field \( F \) and let \( J_A \) denote the Jacobson radical of \( A \).
\index{\( J_A \) : Ideal}%
Recall that \( A \) is local if \( A/J_A \) is a division algebra.
Also \( F \) is a splitting field for \( A \), or \( A \) splits over \( F \), if \( A/J_A \) is isomorphic to a direct product of matrix algebras over \( F \).

Let \( B \) and \( C \) be local algebras over a splitting field \( F \).
Let \( \epsilon_C \colon C \to F \) denote the projection along \( J_C \), called the augmentation map.
\index{\( \epsilon_C \) : Homomorphism}%
Then there are canonical homomorphisms
\[
  \iota_B \colon B \to B \otimes C
  \quad\text{and}\quad
  \pi_B \colon B \otimes C \to B
\]
defined by \( \iota_B(b) = b \otimes 1 \) and \( \pi_B(b \otimes c) = b\epsilon_C(c) \).
\index{\( \iota_B \) : Homomorphism}%
\index{\( \pi_B \) : Homomorphism}%
\index{\( b \) : Element}%
\index{\( c \) : Element}%
Define \( \iota_C \colon C \to B \otimes C \) and \( \pi_C \colon B \otimes C \to C \) similarly.
Let \( T_B \) denote the cotangent space \( J_B/{J_B}^2 \).
\index{\( T_B \) : Vector Space}%
An algebra homomorphism \( \psi \colon B \to C \) induces a cotangent map
\[
  T_\psi \colon T_B \to T_C.
\]
\index{\( \psi \) : Homomorphism}%
Hence \( T \) can be viewed as a functor from the category of local algebras to that of vector spaces.
We repeatedly use the facts that
\[
  T_{\iota_B\pi_B} + T_{\iota_C\pi_C} = \id_{T_{B \otimes C}}
\]
and that if \( T_\psi \) is surjective, then so is \( \psi \).

For algebras \( P \) and \( Q \), a homomorphism \( \rho \colon P \to Q \) is called a retraction if there exists a homomorphism \( \sigma \colon Q \to P \) such that \( \rho\sigma = \id_Q \).
\index{\( Q \) : Algebra}%
\index{\( \rho \) : Homomorphism}%
\index{\( \sigma \) : Homomorphism}%
Such a homomorphism \( \sigma \) is called a section of \( \rho \).

\section{Tensor factorization}
First we define two algebras associated to an endomorphism.
\begin{dfn}
  Let \( Q \) be an algebra and let \( \beta \colon Q \to Q \) be an endomorphism.
  \index{\( \beta \) : Homomorphism}%
  By Fitting's lemma, there exists a least \( k \geq 1 \) such that \( Q = \Im \beta^k \oplus \Ker \beta^k \).
  \index{\( k \) : Integer}%
  We write
  \[
    X_\beta = \Im \beta^k
    \quad\text{and}\quad
    \varpi_{X_\beta} \colon Q \to X_\beta
  \]
  \index{\( X_\beta \) : Algebra}%
  \index{\( \varpi_{X_\beta} \) : Homomorphism}%
  for the projection along \( \Ker \beta^k \).
  Also we write
  \[
    Y_\beta = Q/Q \beta^k(J_Q) Q
    \quad\text{and}\quad
    \varpi_{Y_\beta} \colon Q \to Y_\beta
  \]
  for the canonical projection.
  \index{\( Y_\beta \) : Algebra}%
  \index{\( \varpi_{Y_\beta} \) : Homomorphism}%
\end{dfn}

The following is assumed throughout this section.
\begin{hyp}
  \label{hyp:common}
  Let \( B \), \( C \), \( Q \) be local algebras over a splitting field.
  Let \( \rho \colon B \otimes C \to Q \) be a retraction with a section \( \sigma \colon Q \to B \otimes C \).
  Define the homomorphisms \( \beta \colon Q \to Q \) and \( \gamma \colon Q \to Q \) by \( \beta = \rho\iota_B\pi_B\sigma \) and \( \gamma = \rho\iota_C\pi_C\sigma \).
  \index{\( \gamma \) : Homomorphism}%
  Set \( X = X_\beta \) and \( Y = Y_\beta \).
  \index{\( X \) : Algebra}%
  \index{\( Y \) : Algebra}%
  Define the homomorphism \( \alpha \colon Q \to X \otimes Y \) by \( \alpha = (\varpi_X\rho\iota_B \otimes \varpi_Y\rho\iota_C)\sigma \).
  \index{\( \alpha \) : Homomorphism}%
\end{hyp}
The commutative diagram below summarizes the algebras and homomorphisms in \cref{hyp:common}.
The aim of this section is to prove that \( \alpha \colon Q \to X \otimes Y \) is bijective.
\[
  \begin{tikzcd}[sep=huge]
  B \arrow[d, "\rho\iota_B"'] & B \otimes C \arrow[l, "\pi_B"']\arrow[r, "\pi_C"]                                                                             & C \arrow[d, "\rho\iota_C"] \\
  Q \arrow[d, "\varpi_X"']    & Q \arrow[l, "\beta" description]\arrow[u, "\sigma" description]\arrow[r, "\gamma" description]\arrow[d, "\alpha" description] & Q \arrow[d, "\varpi_Y"] \\
  X                           & X \otimes Y \arrow[l, "\pi_X"]\arrow[r, "\pi_Y"']                                                                             & Y
  \end{tikzcd}
\]

\begin{lem}
  \label{lem:ImKer}
  Under \cref{hyp:common}, the following hold for some \( k \geq 1 \).
  \begin{itemize}
    \item[(a)] \( T_Q = \Im T_{\beta^k} + \Ker T_{\beta^k} \).
    \item[(b)] \( T_{\varpi_X\beta}(\Im T_{\beta^k}) = T_X \).
    \item[(c)] \( T_{\varpi_Y\gamma}(\Im T_{\beta^k}) = 0 \).
    \item[(d)] \( T_{\varpi_X\beta}(\Ker T_{\beta^k}) = 0 \).
    \item[(e)] \( T_{\varpi_Y\gamma}(\Ker T_{\beta^k}) = T_Y \).
  \end{itemize}
\end{lem}
\begin{proof}
  By Fitting's lemma, there exists a least \( k \geq 1 \) such that \( Q = \Im \beta^k \oplus \Ker \beta^k \).
  Since \( T_{\iota_B\pi_B} + T_{\iota_C\pi_C} = \id_{T_{B \otimes C}} \) and \( T_{\rho\sigma} = \id_{T_Q} \), we have
  \[
    T_\beta + T_\gamma
    = T_\rho(T_{\iota_B\pi_B} + T_{\iota_C\pi_C})T_{\sigma}
    = \id_{T_Q}.
  \]

  (a)
  This follows from \( q = \varpi_X(q) + (q - \varpi_X(q)) \) for \( q \in J_Q \).
  \index{\( q \) : Element}%

  (b)
  This is clear from \( \varpi_X\beta^{k + 1}(J_Q) = J_X \).

  (c)
  Since \( \varpi_Y\beta^k(J_Q) = 0 \), this follows from \( T_{\varpi_Y\gamma\beta^k} = T_{\varpi_Y}(\id_{T_Q} - T_\beta)T_{\beta^k} = 0 \).

  (d)
  Since \( \beta \) is bijective on \( X \), this follows from \( \Ker T_{\beta^k} \subseteq \Ker T_{\varpi_X\beta^k} = \Ker T_{\varpi_X\beta} \).

  (e)
  Since \( T_\gamma = \id_{T_Q} - T_\beta \), we have \( T_\gamma(\Ker T_{\beta^k}) = \Ker T_{\beta^k} \).
  Then \( \varpi_Y\beta^k(J_Q) = 0 \) and (a) yield \( T_{\varpi_Y\gamma}(\Ker T_{\beta^k}) = T_{\varpi_Y}(\Ker T_{\beta^k}) = T_{\varpi_Y}(T_Q) = T_Y \).
\end{proof}

\begin{lem}
  \label{lem:surj}
  Under \cref{hyp:common}, \( \alpha \colon Q \to X \otimes Y \) is surjective.
\end{lem}
\begin{proof}
  It suffices to prove that \( T_{X \otimes Y} = \Im T_\alpha \).
  Observe that
  \[
    T_\alpha = T_{\iota_X\varpi_X\beta} + T_{\iota_Y\varpi_Y\gamma}
    \quad\text{and}\quad
    T_{\iota_X\pi_X} + T_{\iota_Y\pi_Y} = \id_{T_{X \otimes Y}}.
  \]
  For some \( k \geq 1 \), by \cref{lem:ImKer}, we have
  \begin{align*}
    \Im T_\alpha
    &= T_\alpha(\Im T_{\beta^k}) + T_\alpha(\Ker T_{\beta^k}) \\
    &= T_{\iota_X\varpi_X\beta}(\Im T_{\beta^k}) + T_{\iota_Y\varpi_Y\gamma}(\Ker T_{\beta^k}) \\
    &= \Im T_{\iota_X} + \Im T_{\iota_Y} \\
    &= T_{X \otimes Y}.
    \qedhere
  \end{align*}
\end{proof}

\begin{lem}
  \label{lem:dim}
  Under \cref{hyp:common}, \( \dim Q \leq \dim X \dim Y \).
\end{lem}
\begin{proof}
  \cref{lem:ImKer}(e) implies \( \Im T_{\varpi_Y\gamma} = T_Y \) and hence \( \Im \varpi_Y\gamma = Y \).
  Set \( \ell = \dim Y \) and choose \( q_1, \dotsc, q_\ell \in \gamma(Q) \) such that \( \varpi_Y(q_1), \dotsc, \varpi_Y(q_\ell) \) form a basis of \( Y \).
  \index{\( l \) @ \( \ell \) : Integer}%
  \index{\( q_1, \dotsc, q_\ell \) : Elements}%
  By definition of \( X \) and \( Y \),
  \[
    Q = X q_1 + \dotsb + X q_\ell + Q J_X Q.
  \]
  Since every element of \( J_X \subseteq \rho\iota_B(B) \) commutes with \( q_1, \dotsc, q_\ell \in \rho\iota_C(C) \) and \( J_Q \) is nilpotent, iterating the above yields
  \begin{align*}
    Q
    &= X q_1 + \dotsb + X q_\ell + Q J_X Q \\
    &= X q_1 + \dotsb + X q_\ell + (Q J_X Q)^2 \\
    &\vdotswithin{=} \\
    &= X q_1 + \dotsb + X q_\ell.
  \end{align*}
  Hence \( \dim Q \leq \dim X \dim Y \).
\end{proof}

\begin{prop}
  \label{prop:factor}
  Under \cref{hyp:common}, \( Q \cong X \otimes Y \).
\end{prop}
\begin{proof}
  The homomorphism \( \alpha \colon Q \to X \otimes Y \) is an isomorphism by \cref{lem:surj,lem:dim}.
  Hence \( Q \cong X \otimes Y \).
\end{proof}

\section{First tensor factor}
The previous section gives a tensor factorization of a retract.
We use it to show that a given irreducible tensor factor is a homomorphic image of some tensor factor in another tensor factorization.
\begin{lem}
  \label{lem:avoid}
  Let \( B \), \( C \), \( Q \) be local algebras over a splitting field.
  Assume that \( Q \) is irreducible.
  If \( \rho \colon B \otimes C \to Q \) is a retraction, then \( \rho\iota_B \colon B \to Q \) or \( \rho\iota_C \colon C \to Q \) is a retraction.
\end{lem}
\begin{proof}
  Let \( \sigma \colon Q \to B \otimes C \) be a section of \( \rho \).
  Define the homomorphisms \( \beta \colon Q \to Q \) and \( \gamma \colon Q \to Q \) by \( \beta = \rho\iota_B\pi_B\sigma \) and \( \gamma = \rho\iota_C\pi_C\sigma \).
  Note that \( T_Q \) is non-zero as \( Q \) is non-trivial.
  Since \( T_{\iota_B\pi_B} + T_{\iota_C\pi_C} = \id_{T_{B \otimes C}} \) and \( T_{\rho\sigma} = \id_{T_Q} \), we have
  \[
    T_\beta + T_\gamma
    = T_\rho(T_{\iota_B\pi_B} + T_{\iota_C\pi_C})T_{\sigma}
    = \id_{T_Q}.
  \]
  In particular, \( T_\beta \) and \( T_\gamma \) commute and at least one of them is not nilpotent.
  We may assume that it is \( T_\beta \).

  By \cref{prop:factor}, \( X_\beta \) is a tensor factor of \( Q \).
  Since \( T_\beta \) is not nilpotent and \( Q \) is irreducible, \( X_\beta \) is non-trivial and isomorphic to \( Q \).
  Then \( \beta \) is bijective and a section of \( \rho\iota_B \) is given by \( \pi_B\sigma\beta^{-1} \).
\end{proof}

\begin{prop}
  \label{prop:first}
  Let \( P_1, \dotsc, P_n \), \( Q \), \( R \) be local algebras over a splitting field.
  \index{\( R \) : Algebra}%
  Assume that \( Q \) is irreducible and there exists an isomorphism
  \[
    \phi \colon P_1 \otimes \dotsb \otimes P_n \to Q \otimes R.
  \]
  Then \( \pi_Q\phi\iota_B \colon B \to Q \) is surjective for some \( B \) among \( P_1, \dotsc, P_n \).
\end{prop}
\index{\( \phi \) : Homomorphism}%
\begin{proof}
  Define the homomorphisms \( \rho \colon P_1 \otimes \dotsb \otimes P_n \to Q \) and \( \sigma \colon Q \to P_1 \otimes \dotsb \otimes P_n \) by \( \rho = \pi_Q\phi \) and \( \sigma = \phi^{-1}\iota_Q \).
  Then \( \rho \) is a retraction with a section \( \sigma \).
  Since \( Q \) is irreducible, applying \cref{lem:avoid} recursively yields an appropriate \( B \).
\end{proof}

\section{Other tensor factors}
We show that a surjective homomorphism between a pair of tensor factors yields a surjective homomorphism between the other pair in the opposite direction.
\begin{prop}
  \label{prop:other}
  Let \( B \), \( C \), \( Q \), \( R \) be local algebras over a splitting field.
  Assume that there exists an isomorphism
  \[
    \phi \colon B \otimes C \to Q \otimes R.
  \]
  If \( \pi_Q\phi\iota_B \colon B \to Q \) is surjective, then so is \( \pi_C\phi^{-1}\iota_R \colon R \to C \).
\end{prop}
\begin{proof}
  It suffices to prove that \( T_C = \Im T_{\pi_C\phi^{-1}\iota_R} \).
  Since \( Q = \Im \pi_Q\phi\iota_B \) by assumption, \( T_Q = \Im T_{\pi_Q\phi\iota_B} \).
  Taking preimages of both sides under \( T_{\pi_Q} \colon T_{Q \otimes R} \to T_Q \) yields
  \begin{align*}
    T_{Q \otimes R}
    &= \Im T_{\phi\iota_B} + \Ker T_{\pi_Q} \\
    &= \Im T_{\phi\iota_B} + \Im T_{\iota_R}.
  \end{align*}
  Applying \( T_{\pi_C\phi^{-1}} \colon T_{Q \otimes R} \to T_C \) yields
  \begin{align*}
    & T_{\pi_C\phi^{-1}}(T_{Q \otimes R}) = \Im T_{\pi_C} = T_C, \\
    & T_{\pi_C\phi^{-1}}(\Im T_{\phi\iota_B}) = \Im T_{\pi_C\iota_B} = 0, \\
    & T_{\pi_C\phi^{-1}}(\Im T_{\iota_R}) = \Im T_{\pi_C\phi^{-1}\iota_R}.
  \end{align*}
  Hence \( T_C = \Im T_{\pi_C\phi^{-1}\iota_R} \).
\end{proof}

\section{Proof of uniqueness}
With the propositions of the previous two sections, a comparison of dimensions completes the proof.
\label{sec:uniq}
\begin{prop}
  \label{prop:ind}
  Let \( P_1, \dotsc, P_n \), \( Q \), \( R \) be local algebras over a splitting field.
  Assume that \( Q \) is irreducible, \( \dim P_i \leq \dim Q \) for every \( i \) and
  \[
    P_1 \otimes \dotsb \otimes P_n \cong Q \otimes R.
  \]
  Then, after reordering indices, we have
  \[
    P_1 \cong Q
    \quad\text{and}\quad
    P_2 \otimes \dotsb \otimes P_n \cong R.
  \]
\end{prop}
\begin{proof}
  Let \( \phi \colon P_1 \otimes \dotsb \otimes P_n \to Q \otimes R \) be an isomorphism.
  Since \( Q \) is irreducible, after reordering indices, \( \pi_Q\phi\iota_B \colon B \to Q \) is surjective with \( B = P_1 \) by \cref{prop:first}.
  Set \( C = P_2 \otimes \dotsb \otimes P_n \).
  Then \( \pi_C\phi^{-1}\iota_R \colon R \to C \) is also surjective by \cref{prop:other}.

  Since \( \dim B \leq \dim Q \) and \( \dim C \geq \dim R \) by assumption, the surjective homomorphisms \( B \to Q \) and \( R \to C \) are isomorphisms.
\end{proof}

\begin{proof}[Proof of \cref{main:thm}]
  We prove uniqueness by induction on \( \min(n, m) \geq 0 \).
  The base case \( \min(n, m) = 0 \) is clear.

  For the induction step \( \min(n, m) \geq 1 \), we may assume that \( Q_1 \) has the maximum dimension among the tensor factors so that \( \dim P_i \leq \dim Q_1 \) for every \( i \).
  Recall that if an algebra is local and splits, then the same holds for tensor factors.
  Since \( Q_1 \) is irreducible, after reordering indices, we have
  \[
    P_1 \cong Q_1
    \quad\text{and}\quad
    P_2 \otimes \dotsb \otimes P_n \cong Q_2 \otimes \dotsb \otimes Q_m
  \]
  by \cref{prop:ind}.
  Hence uniqueness follows from the induction hypothesis.
\end{proof}

\begin{rmk}[Commutativity]
  \label{rmk:comm}
  Commutativity is not explicitly stated in Horst's paper~\cite{Ho87}.
  However, the proof of \cite[Lemma~3]{Ho87} appears to rely on Krull's intersection theorem.
  Since Jacobson's conjecture remains open, we regard the algebras studied in~\cite{Ho87} as commutative.
  This interpretation is also compatible with the fact that the paper is assigned classification codes for commutative algebra in MathSciNet and zbMATH Open.
\end{rmk}

\begin{rmk}[Characteristic]
  \label{rmk:char}
  N\"usken \cite[p.~527]{Nu02} already noted that some intermediate results of Horst~\cite{Ho87} fail in positive characteristic.
  For completeness, we record counterexamples.
  Let \( F \) be a field of positive characteristic \( p \) and let \( G \) be a cyclic group generated by an element \( g \) of order \( p \).
  \index{\( G \) : Group}%
  \index{\( g \) : Element}%
  \index{\( p \) : Integer}%
  Set \( A = B = FG \) and define the homomorphism \( \phi \colon A \to A \otimes B \) by \( \phi(g) = g \otimes g \).
  Then \( \pi_A \phi \colon A \to A \) is an automorphism and \( \Ker \pi_B\phi \neq J_A \).
  In particular, \cite[Lemma~4]{Ho87} fails without assuming characteristic zero.
\end{rmk}

\appendix
\crefalias{section}{appendix}
\section{Application}
\label{appx:MIP}
Applications of tensor factorization were already given in~\cite{Ho87}.
Horst's motivation for studying these problems appears to have come from direct product decompositions of germs of complex analytic spaces.
Our study, by contrast, grew out of joint work with del R\'io, Garc\'ia-Lucas, Margolis and Stanojkovski on the modular isomorphism problem.
In this problem, it is essential to study tensor factorization without assuming commutativity or characteristic zero.

Let \( F \) be a field of positive characteristic \( p \) and let \( G \) and \( H \) be finite \( p \)-groups.
\index{\( H \) : Group}%
Recall that the modular group algebra \( FG \) is local and splits.
The modular isomorphism problem asks whether \( FG \cong FH \) implies \( G \cong H \).
For details, see the survey by Margolis~\cite{M22} and the references therein.
The tensor product of modular group algebras has been studied in~\cite{CK95,GL24,GLdRS26,MSS23}.
For a finite \( p \)-group \( K \), \cref{main:cor} yields a new reduction
\[
  F[G \times K] \cong F[H \times K] \implies FG \cong FH.
\]
\index{\( K \) : Group}%

This type of reduction of the modular isomorphism problem appeared first in joint work with Margolis and Stanojkovski~\cite{MSS23}.
Let \( G_\Elem \) denote an elementary abelian direct factor of \( G \) of maximum order.
\index{\( G_\Elem \) : Group}%
Then \( G \cong G/G_\Elem \times G_\Elem \) and \( G_\Elem \) is uniquely determined up to isomorphism.
Margolis et al.\ \cite[Theorem~4.1]{MSS23} proved that \( FG \cong FH \) if and only if \( F[G/G_\Elem] \cong F[H/H_\Elem] \) and \( FG_\Elem \cong FH_\Elem \).

It is also natural to consider abelian direct factors.
Let \( G_\Ab \) denote an abelian direct factor of \( G \) of maximum order.
\index{\( G_\Ab \) : Group}%
Then \( G \cong G/G_\Ab \times G_\Ab \) and \( G_\Ab \) is uniquely determined up to isomorphism.
Over the prime field \( \FF_p \), Garc\'ia-Lucas~\cite[Theorem~A]{GL24} proved that \( \FF_p G \cong \FF_p H \) if and only if \( \FF_p [G/G_\Ab] \cong \FF_p [H/H_\Ab] \) and \( \FF_p G_\Ab \cong \FF_p H_\Ab \).
\index{\( F_p \) @ \( \FF_p \) : Field}%
It is also proved in \cite[Proposition~C]{GL24} that
\[
  FG \cong FH \implies FG_\Ab \cong FH_\Ab
\]
without assuming \( F = \FF_p \).
Garc\'ia-Lucas \cite[Question~D]{GL24} then asked if the same holds for \( G/G_\Ab \).
Using the cancellation law, we provide a positive answer to this question.

\begin{mainthm}
  Let \( F \) be a field of positive characteristic \( p \) and let \( G \) and \( H \) be finite \( p \)-groups.
  Then
  \[
    FG \cong FH \implies F[G/G_\Ab] \cong F[H/H_\Ab].
  \]
\end{mainthm}
\begin{proof}
  It follows from \cite[Proposition~C]{GL24} that \( FG_\Ab \cong FH_\Ab \).
  Since
  \[
    F[G/G_\Ab] \otimes FG_\Ab \cong FG \cong FH \cong F[H/H_\Ab] \otimes FH_\Ab,
  \]
  \cref{main:cor} yields \( F[G/G_\Ab] \cong F[H/H_\Ab] \).
\end{proof}

\printbibliography

@article {CK95,
    AUTHOR = {Carlson, Jon F. and Kov\'acs, L. G.},
     TITLE = {Tensor factorizations of group algebras and modules},
   JOURNAL = {J. Algebra},
  FJOURNAL = {Journal of Algebra},
    VOLUME = {175},
      YEAR = {1995},
     PAGES = {385--407},
       DOI = {10.1006/jabr.1995.1193},
}

@article {GL24,
    AUTHOR = {Garc\'ia-Lucas, Diego},
     TITLE = {The modular isomorphism problem and abelian direct factors},
   JOURNAL = {Mediterr. J. Math.},
  FJOURNAL = {Mediterranean Journal of Mathematics},
    VOLUME = {21},
      YEAR = {2024},
     PAGES = {Paper No. 18, 21},
       DOI = {10.1007/s00009-023-02557-1},
}

@article {GLdRS26,
    AUTHOR = {Garc\'ia-Lucas, Diego and del R\'io, \'Angel and Sakurai,
              Taro},
     TITLE = {On commutative tensor factors of group algebras},
   JOURNAL = {Algebr. Represent. Theory},
  FJOURNAL = {Algebras and Representation Theory},
    VOLUME = {29},
      YEAR = {2026},
     PAGES = {153--160},
       DOI = {10.1007/s10468-026-10379-4},
}

@article {Ho87,
    AUTHOR = {Horst, Camilla},
     TITLE = {A cancellation theorem for {A}rtinian local algebras},
   JOURNAL = {Math. Ann.},
  FJOURNAL = {Mathematische Annalen},
    VOLUME = {276},
      YEAR = {1987},
     PAGES = {657--662},
       DOI = {10.1007/BF01456993},
}

@article {M22,
    AUTHOR = {Margolis, Leo},
     TITLE = {The modular isomorphism problem: a survey},
   JOURNAL = {Jahresber. Dtsch. Math.-Ver.},
  FJOURNAL = {Jahresbericht der Deutschen Mathematiker-Vereinigung},
    VOLUME = {124},
      YEAR = {2022},
     PAGES = {157--196},
       DOI = {10.1365/s13291-022-00249-5},
}

@article {MSS23,
    AUTHOR = {Margolis, Leo and Sakurai, Taro and Stanojkovski, Mima},
     TITLE = {Abelian invariants and a reduction theorem for the modular
              isomorphism problem},
   JOURNAL = {J. Algebra},
  FJOURNAL = {Journal of Algebra},
    VOLUME = {636},
      YEAR = {2023},
     PAGES = {533--559},
       DOI = {10.1016/j.jalgebra.2023.08.035},
}

@article {Nu02,
    AUTHOR = {N\"usken, Michael},
     TITLE = {Unique tensor factorization of loop-resistant algebras over a
              field of finite characteristic},
   JOURNAL = {J. Algebra},
  FJOURNAL = {Journal of Algebra},
    VOLUME = {251},
      YEAR = {2002},
     PAGES = {509--528},
       DOI = {10.1006/jabr.2001.9126},
}


\end{document}